\documentclass[review]{elsarticle}
\usepackage{lineno,hyperref}
\usepackage{graphicx}
\usepackage{amsfonts}
\usepackage{amssymb, amsmath, mathrsfs, amsthm}
\usepackage{booktabs}
\usepackage{array}
\usepackage{caption}
\usepackage{amsmath}
\usepackage{mathrsfs}
\usepackage{lineno,hyperref}
\usepackage{graphicx}
\usepackage{amsfonts, amssymb}
\usepackage{booktabs}
\usepackage{array}
\usepackage{caption}
\usepackage{amsmath,mathrsfs}
\usepackage{dsfont}
\usepackage{curves}
\usepackage{graphicx}
\usepackage{pifont}
\usepackage{cases}
\usepackage[all,pdf]{xy}
\usepackage{caption}
\allowdisplaybreaks[4]

\newtheorem{theorem}{Theorem}

\newtheorem{pro}[theorem]{Problem}
\newtheorem{remark}[theorem]{Remark}

\newtheorem{cla}{Claim}

\modulolinenumbers[5]

\begin{document}
\begin{frontmatter}

\title{Sufficient conditions for the existence of path-factors with given properties}
\author[address]{Hui Qin}
\ead{qinhui$_-$0213@163.com}

\author[address1]{Guowei Dai}
\ead{guoweidai1990@gmail.com}

\author[mymainaddress]{Yuan Chen\corref{mycorrespondingauthor}}
\cortext[mycorrespondingauthor]{Corresponding author}
\ead{chenyuanmath@hotmail.com}

\author[address1]{Ting Jin}
\ead{tingjin@njfu.edu.cn}

\author[address2]{Yuan Yuan}
\ead{yyuan@hainanu.edu.cn}

\address[address] {School of mathematics and big data, Anhui University of Science and Technology,
Huainan, 232001, P.R. China.}
\address[address1] {College of Science, Nanjing Forestry University, Nanjing, 210037, P.R. China.}
\address[mymainaddress]
{Research Center of Nonlinear Science, School of Mathematical and Physical Sciences, Wuhan Textile University, Wuhan, 430073, P.R. China. }
\address[address2] {School of Science, Hainan University, Haikou, 570228, P.R. China.}

\begin{abstract}
A spanning subgraph $H$ of a graph $G$ is called a $P_{\geq k}$-factor of $G$ if every component of $H$ is isomorphic to a path of order at least $k$, where $k\geq2$ is an integer.
A graph $G$ is called a $(P_{\geq k},l)$-factor critical graph if $G-V'$ contains a $P_{\geq k}$-factor for any $V'\subseteq V(G)$ with $|V'|=l$.
A graph $G$ is called a $(P_{\geq k},m)$-factor deleted graph if $G-E'$ has a $P_{\geq k}$-factor for any $E'\subseteq E(G)$ with $|E'|=m$.
Intuitively, if a graph is dense enough, it will have a $P_{\geq 3}$-factor.
In this paper, we give some sufficient conditions for a graph to be a $(P_{\geq 3},l)$-factor critical graph or a $(P_{\geq 3},m)$-factor deleted graph.
In this paper, we demonstrate that
(i) $G$ is a $(P_{\geq 3},l)$-factor critical graph if its sun toughness $s(G)>\frac{l+1}{3}$ and $\kappa(G)\geq l+2$.
(ii) $G$ is a $(P_{\geq 3},l)$-factor critical graph if its degree sum $\sigma_3(G)\geq n+2l$ and $\kappa(G)\geq l+1$.
(iii) $G$ is a $(P_{\geq 3},m)$-factor deleted graph if its sun toughness $s(G)\geq \frac{m+1}{m+2}$ and $\kappa(G)\geq 2m+1$.
(iv) $G$ is a $(P_{\geq 3},m)$-factor deleted graph if its degree sum $\sigma_3(G)\geq n+2m$ and $\kappa(G)\geq 2m+1$.
\end{abstract}


\end{frontmatter}

%\linenumbers

\section{Introduction}
The underlying topology in parallel machines is a graph, in which processors are represented by vertices and links between processors are represented by edges. Such graphs are interconnection networks.
Matching parameters have important applications on measuring the reliability of interconnection networks; see \cite{ChengL, ChengL2}.
Graph factors are generalizations of matchings; see \cite{LDHKPK07} for some of their applications.
The path factor of a graph is a hot research topic in structural graph theory, and can be viewed as a generalization of the cardinality matching problem.
It is not only of profound theoretical significance, but also of extensive applied value in information science, management science, and other fields.
For example, certain file transfer problems and scheduling problems in networks can be converted into path-factor problems in graphs.

In the paper, we deal with only finite simple graph, unless explicitly stated.
We refer to \cite{Bondy1982} for the notation and terminologies not defined here.
Let $G = (V(G), E(G))$ be a simple graph, where $V(G)$ and $E(G)$ denote the vertex set and the edge set of $G$, respectively.
A subgraph $H$ of $G$ is called a spanning subgraph of $G$ if $V(H)=V(G)$ and $E(H)\subseteq E(G)$.
Given a vertex $v\in V(G)$, let $N_{G}(v)$ be the set of vertices adjacent to $v$ in $G$ and $d_{G}(v)=|N_{G}(v)|$ be the degree of $v$ in $G$.
The number of connected components of a graph $G$ is denoted by $\omega(G)$.

A subgraph $H$ of $G$ is called an induced subgraph of $G$ if every pair of vertices in $H$
which are adjacent in $G$ are also adjacent in $H$.
For any subset $S\subseteq V(G)$, let $G[S]$ denote the subgraph of $G$ induced by $S$,
and $G-S:=G[V(G)\setminus S]$ is the resulting graph after deleting the vertices of $S$ from $G$.
For any $M\subseteq E(G)$, we use $G-M$ to denote the subgraph obtained from $G$ by deleting $M$.
Especially, we write
$G-x=G-\{x\}$ for $S=\{x\}$ and $G-e=G-\{e\}$ for $M=\{e\}$.

For a family of connected graphs $\mathcal{F}$, a spanning subgraph $H$ of a graph $G$ is called an $\mathcal{F}$-factor of $G$ if each component of $H$ is isomorphic to some graph in $\mathcal{F}$.
A spanning subgraph $H$ of a graph $G$ is called a $P_{\geq k}$-factor of $G$ if every component of $H$ is isomorphic to a path of order at least $k$.
For example, a $P_{\geq 3}$-factor means a graph factor in which every component is a path of order at least three.
A graph $G$ is called a $(P_{\geq k},l)$-factor critical graph if $G-V'$ contains a $P_{\geq k}$-factor for any $V'\subseteq V(G)$ with $|V'|=l$.
A graph $G$ is called a $(P_{\geq k},m)$-factor deleted graph if $G-E'$ has a $P_{\geq k}$-factor for any $E'\subseteq E(G)$ with $|E'|=m$.
Especially, a $(P_{\geq k},m)$-factor deleted graph is simply called a $P_{\geq k}$-factor deleted graph if $m=1$.

\begin{figure}
\centerline{\includegraphics[width=10cm]{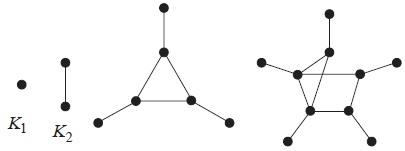}}
\caption{Suns}
\end{figure}

The concept of a sun was introduced by Kaneko \cite{Kaneko2003} as follows (e.g., see Figure 1).
A graph $H$ is called factor-critical if $H-v$ has a 1-factor for each $v\in V(H)$.
Let $H$ be a factor-critical graph and $V(H)=\{v_1,v_2,...,v_n\}$.
By adding new vertices $\{u_1,u_2,...,u_n\}$ together with new edges $\{v_{i}u_{i} : 1 \leq i\leq n\}$ to $H$, the resulting graph is called a sun.
Note that, according to Kaneko \cite{Kaneko2003}, we regard $K_1$ and $K_2$ also as a sun, respectively.
Usually, the suns other than $K_1$ are called big suns.
It is called a sun component of $G$ if the component of $G$ is isomorphic to a sun.
We denote by $sun(G)$ the number of sun components in $G$.

For a connected graph $G$, its \emph{sun~toughness}, denoted by $s(G)$, was defined as follows.
If $G$ is complete, then $s(G)=+\infty$; otherwise,
$$
s(G)=\min\left\{\frac{|X|}{sun(G-X)}: X\subseteq V(G), sun(G-X)\geq2\right\}.
$$

Since Tutte proposed the well-known Tutte 1-factor theorem \cite{Tutte1952}, there are many results on graph factors \cite{Bujtas2020,Dai2020,Kano2008,Lu2020,Lu2017}.
Some related problems about path-factor \cite{Dai3,Egawa2018,Y.Egawa2018,S.Zhou2019}
and path-factor covered graphs \cite{Dai,Dai2,Dai2021,Zhang2009,Zhou2019,Zhou2017,Zhou2020}
have also attracted a great deal of attention.
More results on graph factors are referred to the survey papers and books \cite{Akiyama1985,Plummer2007,Yu2009}.

Recently, Kaneko \cite{Kaneko2003} gave a characterization for a graph with a $P_{\geq 3}$-factor,
for which Kano et al. \cite{Kano2004} presented a simpler proof.

\begin{theorem}\label{thm:1}(Kaneko \cite{Kaneko2003})
A graph $G$ has a $P_{\geq 3}$-factor if and only if $sun(G-X)\leq 2|X|$ for all $X\subseteq V(G)$.
\end{theorem}

A claw is a graph isomorphic to $K_{1,3}$.
A graph $G$ is said to be claw-free if there is no induced subgraph of $G$ isomorphic to $K_{1,3}$.
For a 2-connected claw-free graph $G$,
Kelmans \cite{Kelmans2011} obtained a sufficient condition for the existence of $P_3$-factors in $G-x$ for any $x\in V(G)$
and in $G-e$ for any $e\in E(G)$, respectively.
Motivated by the two results, we naturally consider the more general problem as following:
\begin{pro}
Does $G-V'$ have a $P_{\geq 3}$-factor for any $V'\subseteq V(G)$ with $|V'|=l$, or is a graph $G$ a $(P_{\geq 3},l)$-factor critical graph?
\end{pro}
\begin{pro}
Does $G-E'$ have a $P_{\geq 3}$-factor for any $E'\subseteq E(G)$ with $|E'|=m$, or is a graph $G$ a $(P_{\geq 3},m)$-factor deleted graph?
\end{pro}

%These definitions of $(P_{\geq 3},l)$-factor critical graph and $(P_{\geq 3},m)$-factor deleted graph
%have network application backgrounds as follows.
%\begin{itemize}
  %\item $(P_{\geq 3},l)$-factor critical graph (resp. $(P_{\geq 3},m)$-factor deleted graph)
  %reflects whether the remaining subgraph can continue to transmit data when some channels (resp. sites) are damaged due to network attacks.
  %\item In the self-definition network, the path of data transmission can be re-calculated based on real-time network monitoring.
  %The $l$ vertices (resp. $m$ edges) is corresponding to the $l$ congested sites (resp. $m$ congested channels) at a certain moment,
  %and $(P_{\geq 3},l)$-factor critical graph (resp. $(P_{\geq 3},m)$-factor deleted graph) reflects that SDN can recalculate the data
  %transmission path after the congested channels (resp. sites) are eliminated at a certain moment.
%\end{itemize}
Suppose $G$ is of order $n\geq 3$. Then clearly if $G$ is Hamiltonian, $G$ has a $P_{\geq 3}$-factor by deleting one edge from a Hamiltonian cycle. So our goal is to replace
``$d_G(u)+d_G(v)\geq n$ for every pair of nonadjacent vertices of $G$'' by a weaker condition of the same flavor. Here, we use the graphic parameter \emph{degree sum}.
Let $G$ be a graph containing at least $k$ independent vertices, define the degree sum
$$\sigma_k(G)=\min_{X\subseteq V(G)}\Big\{\sum_{x\in X}d_G(x): \mathrm{the~set} X~\mathrm{is~independent~and~contains}~k~\mathrm{vertices}\Big\}.$$
Note that when $k=2$, this corresponds to taking the minimum of $d_G(u)+d_G(v)$ over every pair of nonadjacent vertices of $G$, part of the crux of the statement of Ore's Theorem.

In this paper, in terms of sun toughness and degree sum, we investigate the graphs admitting path-factors in some special settings.
As main results, we obtain two sufficient conditions for graphs to be $(P_{\geq 3},l)$-factor critical graphs and $(P_{\geq 3},m)$-factor deleted graphs.

\section{$(P_{\geq3},l)$-factor critical graph}
\label{sec:2}

\begin{theorem}\label{thm:2}
Let $l\geq1$ be an integer.
A graph $G$ with $\kappa(G)\geq l+2$ is a $(P_{\geq3},l)$-factor critical graph if $s(G)>\frac{l+1}{3}$.
\end{theorem}
\begin{proof}
If $G$ is a complete graph, then it is easily seen that $G$ is a $(P_{\geq3},l)$-factor critical graph.
Hence, we assume that $G$ is not a complete graph.

For any $V'\subseteq V(G)$ with $|V'|=l$, we write $G'=G-V'$.
To verify the theorem, we only need to prove that $G'$ contains a $P_{\geq 3}$-factor.
On the contrary, we assume that $G'$ has no $P_{\geq 3}$-factor.
Then by Theorem \ref{thm:1}, there exists a subset $X\subseteq V(G')$ such that $sun(G'-X)>2|X|$.
In terms of the integrality of $sun(G'-X)$, we obtain that
\begin{equation}\label{eqn1-1}
sun(G'-X)\geq2|X|+1.
\end{equation}

\begin{cla}\label{cla1-1}
$X\neq\emptyset$ and $sun(G'-X)\geq3$.
\end{cla}
\begin{proof}
On the contrary, we assume that $X=\emptyset$.
On the one hand, since $\kappa(G')\geq\kappa(G)-l\geq 2$, $G'$ is a 2-connected graph such that $|G'|\geq3$.
Thus, $G'$ is not a sun, i.e., $sun(G')=0$.
On the other hand, it follows from (\ref{eqn1-1}) that $sun(G')=sun(G'-X)\geq2|X|+1=1$, which contradicts $sun(G')=0$.
Hence, $X\neq\emptyset$ and $|X|\geq1$.
By (\ref{eqn1-1}), we have that $sun(G'-X)\geq2|X|+1\geq3$.
\end{proof}

\vspace{3mm}

Note that $sun(G-V'\cup X)=sun(G'-X)\geq3$.
Combining this with $s(G)>\frac{l+1}{3}$ and the definition of $s(G)$, we obtain that
\begin{eqnarray*}
\frac{l+1}{3}
&<& s(G)
\\
&\leq& \frac{|V'\cup X|}{sun(G-V'\cup X)}
\\
&=& \frac{|X|+l}{sun(G'-X)}
\\
&\leq& \frac{\frac{sun(G'-X)-1}{2}+l}{sun(G'-X)}
\\
&=& \frac{1}{2}+\frac{2l-1}{2sun(G'-X)}
\\
&\leq& \frac{1}{2}+\frac{2l-1}{6}
\\
&=&  \frac{l+1}{3},
\end{eqnarray*}
where the last inequality follows from Claim \ref{cla1-1}.
By (\ref{eqn1-1}), we have that $|X|\leq \frac{sun(G'-X)-1}{2}$ and the third inequality follows.
This is a contradiction and completes the proof of Theorem \ref{thm:2}.
\end{proof}

\begin{remark}
The conditions $s(G)>\frac{l+1}{3}$ and $\kappa(G)\geq l+2$ in Theorem \ref{thm:2} cannot be replaced by
$s(G)\geq\frac{l+1}{3}$ and $\kappa(G)\geq l+1$.
We consider the graph $G=K_{l+1}\vee3K_2$, and choose $V'\subseteq V(K_{l+1})$ with $|V'|=l$.
Then it is easily seen that $s(G)=\frac{l+1}{3}$ and $\kappa(G)=l+1$.
Let $G'=G-V'$.
For $X=V(K_{l+1})\setminus V'$, we have that $sun(G' -X)=3>2=2|X|$.
In view of Theorem \ref{thm:1}, $G'$ has no $P_{\geq 3}$-factor.
Hence, $G$ is not a $(P_{\geq 3},l)$-factor critical graph.
\end{remark}

\begin{theorem}\label{thm2}
Let $l\geq1$ be an integer.
A graph $G$ with $\kappa(G)\geq l+1$ is a $(P_{\geq3},l)$-factor critical graph if $\sigma_3(G)\geq n+2l$.
\end{theorem}
\begin{proof}
If $G$ is a complete graph, then it is easily seen that $G$ is a $(P_{\geq3},l)$-factor critical graph.
Hence, we assume that $G$ is not a complete graph.

For any $V'\subseteq V(G)$ with $|V'|=l$, we write $H=G-V'$.
To verify the theorem, we only need to prove that $H$ contains a $P_{\geq 3}$-factor.

By contradiction, suppose that $H$ admits no $P_{\geq3}$-factor.

\textbf{Claim 1.}~~$n\geq l+7$.

\begin{proof}
By $\sigma_3(H)\geq \sigma_3(G)-3l\geq n-l$, it is easy to verify that $n-l\geq5$.
Let $\{u_1,u_2,u_3\}$ be an independent set of $H$.
If $n-l=5$, let $\{v_1,v_2\}=V(H)\setminus \{u_1,u_2,u_3\}$.
Since $\Sigma_{i=1}^3d_H(u_i)\geq\sigma_3(H)\geq 5$, without loss of generality, we may assume that $d_H(u_1)=d_H(u_2)=2$ and $u_3v_1\in E(H)$.
Then we can obtain a path: $u_1v_2u_2v_1u_3$, which can be viewed as a $P_{\geq3}$-factor of $H$, a contradiction.

Next, we consider the case that $n-l=6$. Let $\{v_1,v_2,v_3\}=V(H)\setminus \{u_1,u_2,u_3\}$.
Since $\Sigma_{i=1}^3d_H(u_i)\geq\sigma_3(H)\geq 6$, without loss of generality, we may assume that $d_H(u_i)\geq 2$ for $i=1,2,3$ or $d_H(u_1)=3,d_H(u_2)\geq2,d_H(u_3)=1$.
\begin{itemize}
  \item[$\bullet$] If, for $i=1,2,3$, $d_H(u_i)\geq 2$ and $|N_H(u_i)\cap N_H(u_j)|=1$
  for every $1\leq i<j\leq3$, then without loss of generality, we may assume that $N_H(u_1)=\{v_1,v_2\},N_H(u_2)=\{v_2,v_3\},N_H(u_3)=\{v_1,v_3\}$.
  Then we can obtain a path: $u_1v_1u_3v_3u_2v_2$, which can be viewed as a $P_{\geq3}$-factor of $G$, a contradiction.
  \item[$\bullet$]  If, for $i=1,2,3$, $d_H(u_i)\geq 2$ and $|N_H(u_i)\cap N_H(u_j)|\geq2$ holds
  for some $1\leq i<j\leq3$, then without loss of generality, we may assume that $\{v_1,v_2\}\subseteq N_H(u_1)\cap N_H(u_2),N_H(u_3)=\{v_2,v_3\}$.
  Then we can obtain a path: $u_1v_1u_2v_2u_3v_3$, which can be viewed as a $P_{\geq3}$-factor of $G$, a contradiction.
  \item[$\bullet$]  If $d_H(u_1)=3,d_H(u_2)\geq2$, and $d_H(u_3)=1$, then without loss of generality, we assume that $N_H(u_1)=\{v_1,v_2,v_3\},N_H(u_2)=\{v_1,v_2\},|N_H(u_3)\cap\{v_2,v_3\}|=1$.
Then we can obtain a path: $u_3v_2u_2v_1u_1v_3$ or $u_3v_3u_1v_1u_2v_2$, which can be viewed as a $P_{\geq3}$-factor of $G$, a contradiction.
\end{itemize}\end{proof}

By Theorem \ref{thm:1} and the integrality of $sun(H-S)$, there exists $S\subseteq V(H)$ such that
\begin{equation}\label{eqn2-1}
sun(H-S)\geq2|S|+1.
\end{equation}

\mbox{}\\
\textbf{Claim 2.}~~$S\neq \emptyset$ and $sun(H-S)\geq3$.

\begin{proof}
Suppose that $S=\emptyset$, then $sun(G)=sun(H-S)\geq1$  by (\ref{eqn2-1}).
On the other hand, as $H$ is connected, we have $sun(H)\leq\omega(H)=1$.
So $H$ is a big sun with $n-l=|V(H)|\geq7$.
Due to the definition of a big sun, there are three distinct vertices of degree one, and the vertices set is denoted by $\{u,v,w\}$.
Obviously, $\{u,v,w\}$ is an independent set of $H$.
So, we have $3=d_{H}(u)+d_{H}(v)+d_{H}(w)\geq\sigma_{3}(H)\geq n-l\geq7$, a contradiction.
Thus we obtain $|S|\geq1$.
This together with (\ref{eqn2-1}) implies that $sun(H-S)\geq2|S|+1\geq3$. \end{proof}

Denote the set of isolated vertices of $H-S$ by $I(H-S)$.
Next, we consider three cases.

\vspace{3mm}

{\bf Case~1}. $i(H-S)\geq3$.

\vspace{2mm}

In this case, we can choose three independent vertices $u,v,w\in I(H-S)$.
It follows that
$d_H(u)+d_H(v)+d_H(w)\geq\sigma_3(H)\geq n-l.$
This together with $N_{H}(u)\cup N_{H}(v)\cup N_{H}(w)\subseteq S$ implies
$$
|S| \geq \max\{d_H(u),d_H(v),d_H(w)\}\geq \frac{\sigma_3(H)}{3}\geq \frac{n-l}{3}.
$$

Combining (\ref{eqn2-1}) and the inequality above, we have that
$n-l\geq|S|+sun(G-S)\geq3|S|+1\geq n-l+1$, which is a contradiction.

\vspace{3mm}

{\bf Case~2}. $i(H-S)=2$.

\vspace{2mm}

By Claim 2, $H-S$ has a sun component containing at least two vertices, denoted by $C$.
Let $u,v$ be two distinct vertices in $I(H-S)$.
Due to the definition of a sun, we can choose a vertex $w\in V(C)$ such that $d_{C}(w)=1$.
Obviously, $\{u,v,w\}$ is independent in $H$.
It follows that
$d_H(u)+d_H(v)+d_H(w)\geq\sigma_3(H)\geq n-l.$
Since $N_{H}(u)\cup N_{H}(v)\subseteq S$ and $d_{S}(w)=d_{H}(w)-1$, we obtain:
\begin{eqnarray*}
|S|
&\geq& \max\{d_S(u),d_S(v),d_S(w)\}
\\
&\geq& \frac{d_S(u)+d_S(v)+d_S(w)}{3}
\\
&=& \frac{\sigma_3(H)-1}{3}
\\
&\geq& \frac{n-l-1}{3}.
\end{eqnarray*}
This together with (\ref{eqn2-1}) implies
\begin{eqnarray*}
n-l
&\geq& |S|+2\times sun(G-S)-i(G-S)
\\
&\geq& |S|+2\times (2|S|+1)-2
\\
&\geq& 5|S|
\\
&\geq& \frac{5n-5l-5}{3},
\end{eqnarray*}
that is, $n-l\leq2$, a contradiction to Claim 1.

\vspace{3mm}

{\bf Case~3}. $i(H-S)\leq1$.

\vspace{2mm}

By Claim 2, $H-S$ has at least three sun components, denoted by $C_1,C_2,C_3$.
We can choose vertex $x_i\in V(C_i)$ such that $d_{C_i}(x_i)\leq1$ for every $i=1,2,3$.
Then $\{x_1,x_2,x_3\}$ is independent in $H$, and thus
\begin{equation}\label{eqn2-3}
\sum_{i=1}^3d_{H}(x_i)\geq\sigma_3(H)\geq n-l.
\end{equation}
Note that $d_{S}(x_i)\geq d_{G}(x_i)-1$ for every $i=1,2,3$.
This together with (\ref{eqn2-3}) implies
\begin{equation}\label{eqn2-4}
|S|\geq\max\{d_S(x_i):i=1,2,3\}\geq\frac{\sum_{i=1}^3d_{S}(x_i)}{3}\geq \frac{n-l}{3}-1.
\end{equation}
Combining (\ref{eqn2-1}) and (\ref{eqn2-4}), we obtain
\begin{eqnarray*}
n-l
&\geq& |S|+2\times sun(G-S)-i(G-S)
\\
&\geq& |S|+2\times (2|S|+1)-i(G-S)
\\
&\geq& 5|S|+1
\\
&\geq& \frac{5n-5l}{3}-4,
\end{eqnarray*}
that is, $n-l\leq6$.
This is a contradiction, and Theorem \ref{thm2} holds.
\end{proof}

\section{$(P_{\geq 3},m)$-factor deleted graphs}
\label{sec:3}

Zhou\cite{Zhou2019} verified that a 2-connected graph $G$ is a $P_{\geq 3}$-factor deleted graph if $s(G)\geq1$.
We extend the above result and give a sufficient condition for a graph being a $(P_{\geq 3},m)$-factor deleted graph.

\begin{theorem}\label{thm:3}
Let $m\geq1$ be an integer.
A graph $G$ with $\kappa(G)\geq 2m+1$ is a $(P_{\geq3},m)$-factor deleted graph if $s(G)\geq\frac{m+1}{m+2}$.
\end{theorem}
\begin{proof}
If $G$ is a complete graph, then it is obvious that $G$ is a $(P_{\geq3},m)$-factor deleted graph by $\lambda(G)\geq\kappa(G)\geq2m+1$
and the definition of $(P_{\geq3},m)$-factor deleted graph.
Hence, we may assume that $G$ is a non-complete graph.

For any $E'\subseteq E(G)$ with $|E'|=m$, we write $G'=G-E'$.
To verify the theorem, we only need to prove that $G'$ has a $P_{\geq 3}$-factor.
By contradiction, we assume that $G'$ has no $P_{\geq 3}$-factor.
Then by Theorem \ref{thm:1}, there exists a subset $X\subseteq V(G')$ such that $sun(G'-X)>2|X|$.
In terms of the integrality of $sun(G'-X)$, we obtain that
\begin{equation}\label{eqn1-2}
sun(G'-X)\geq2|X|+1.
\end{equation}

It follows from $\lambda(G)\geq\kappa(G)\geq2m+1$ that $\lambda(G')\geq\lambda(G)-m\geq m+1\geq2$.
Hence, $sun(G')=0$ by the definition of sun component.
Next, we claim that $X\neq\emptyset$.
Otherwise, $X=\emptyset$, and thus $sun(G')=sun(G'-X)\geq2|X|+1=1$ by (\ref{eqn1-2}), which contradicts $sun(G')=0$.
Therefore, we have that $X\neq\emptyset$ and $|X|\geq1$.

We will distinguish two cases below to completes the proof of Theorem \ref{thm:3}.

\vspace{2mm}

{\bf Case~1}. $G-X$ is not connected.

\vspace{2mm}

Since $G-X$ is not connected, $X$ is a vertex cut set of $G$.
Then we have that $|X|\geq2m+1$ by the connectivity of $G$ that $\kappa(G)\geq2m+1$.
Note that after deleting an edge in a graph, the number of its sun components increases by at most 2.
Hence, we have that
\begin{equation}\label{eqn1-3}
sun(G'-X)=sun(G-E'-X)\leq sun(G-X)+2m.
\end{equation}
Combining (\ref{eqn1-3}) with (\ref{eqn1-2}), we have:
$$2|X|+1\leq sun(G'-X)\leq sun(G-X)+2m.$$
It follows that
\begin{eqnarray*}
sun(G-X)
&\geq& 2|X|+1-2m
\\
&\geq& 2\times(2m+1)+1-2m
\\
&=& 2m+3
\\
&\geq& 5.
\end{eqnarray*}
Then by the definition of sun toughness, we obtain that
\begin{eqnarray*}
\frac{m+1}{m+2}
&\leq& s(G)
\\
&\leq& \frac{|X|}{sun(G-X)}
\\
&\leq& \frac{|X|}{2|X|+1-2m}
\\
&=& \frac{1}{2-\frac{2m-1}{|X|}}
\\
&\leq& \frac{1}{2-\frac{2m-1}{2m+1}}
\\
&=& \frac{2m+1}{2m+3},
\end{eqnarray*}
which is a contradiction to that $m\geq1$.

\vspace{3mm}

{\bf Case~2}. $G-X$ is connected.

\vspace{2mm}

It is easily seen that $sun(G-X)\leq\omega(G-X)=1$ since $G-X$ is connected.
Recall that
\begin{equation}\label{eqn1-4}
sun(G'-X)=sun(G-E'-X)\leq sun(G-X)+2m.
\end{equation}
This together with (\ref{eqn1-2}) implies that $2|X|+1\leq sun(G'-X)\leq sun(G-X)+2m\leq 1+2m$, and thus
\begin{equation}\label{eqn1-5}
|X|\leq m.
\end{equation}
It follows from (\ref{eqn1-5}) and $\kappa(G)\geq2m+1$ that
\begin{equation}\label{eqn1-6}
\lambda(G-X)\geq\kappa(G-X)\geq\kappa(G)-|X|\geq2m+1-m=m+1.
\end{equation}
Using (\ref{eqn1-6}), we obtain that $\lambda(G'-X)\geq\lambda(G-X)-m\geq1$.
Hence, $sun(G'-X)\leq\omega(G'-X)=1<2|X|$, which contradicts (\ref{eqn1-2}).
\end{proof}

\begin{remark}
The sun toughness condition $s(G)\geq\frac{4m+1}{4m+4}$ in Theorem \ref{thm:3} cannot be replaced by $s(G)\geq\frac{2m+1}{3m+3}$.
We consider the graph $G=K_{2m+1}\vee(3m+3)K_2$, and choose $E'\subseteq E(3(m+1)K_2)$ with $|E'|=m$.
Then it is easily seen that $s(G)=\frac{2m+1}{3m+3}$ and $\kappa(G)=2m+1$.
Let $G'=G-E'$.
For $X=V(K_{2m+1})\subseteq V(G' )$, we have that $sun(G' -X)=(2m+3)+2m>2(2m+1)=2|X|$.
In view of Theorem \ref{thm:1}, $G'$ has no $P_{\geq 3}$-factor.
Hence, $G$ is not a $(P_{\geq 3},m)$-factor deleted graph.
\end{remark}

\begin{theorem}\label{thm1-5}
Let $m\geq1$ be an integer.
A graph $G$ with $\kappa(G)\geq 2m+1$ is a $(P_{\geq3},m)$-factor deleted graph if $\sigma_3(G)\geq n+2m$.
\end{theorem}
\begin{proof}
If $G$ is a complete graph, then it is obvious that $G$ is a $(P_{\geq3},m)$-factor deleted graph by $\lambda(G)\geq\kappa(G)\geq2m+1$
and the definition of $(P_{\geq3},m)$-factor deleted graph.
Hence, we may assume that $G$ is a non-complete graph.

For any $E'\subseteq E(G)$ with $|E'|=m$, we write $G'=G-E'$.
To verify the theorem, we only need to prove that $G'$ has a $P_{\geq 3}$-factor.
By contradiction, we assume that $G'$ has no $P_{\geq 3}$-factor.
Then by Theorem \ref{thm:1}, there exists a subset $X\subseteq V(G')$ such that $sun(G'-X)>2|X|$.
In terms of the integrality of $sun(G'-X)$, we obtain that
\begin{equation}\label{eqn3-2}
sun(G'-X)\geq2|X|+1.
\end{equation}

It follows from $\lambda(G)\geq\kappa(G)\geq2m+1$ that $\lambda(G')\geq\lambda(G)-m\geq m+1\geq2$.
Hence, $sun(G')=0$ by the definition of sun component.
Next, we claim that $X\neq\emptyset$.
Otherwise, $X=\emptyset$, and thus $sun(G')=sun(G'-X)\geq2|X|+1=1$ by (\ref{eqn3-2}), which contradicts $sun(G')=0$.
Therefore, we have that $X\neq\emptyset$ and $|X|\geq1$.

Denote by $Sun(G-S)$ the set of sun components of $G-S$.
We will distinguish two cases below to completes the proof of Theorem \ref{thm1-5}.

\vspace{2mm}

{\bf Case~1}. $G-X$ is not connected.

\vspace{2mm}

Since $G-X$ is not connected, $X$ is a vertex cut set of $G$.
Then we have that $|X|\geq2m+1$ by the connectivity of $G$ that $\kappa(G)\geq2m+1$.
Note that after deleting an edge in a graph, the number of its sun components increases by at most 2.
Hence, we have that
\begin{equation}\label{eqn3-3}
sun(G'-X)=sun(G-E'-X)\leq sun(G-X)+2m.
\end{equation}
Combining (\ref{eqn3-3}) with (\ref{eqn3-2}), we have:
$$2|X|+1\leq sun(G'-X)\leq sun(G-X)+2m.$$
It follows that
\begin{eqnarray*}
sun(G-X)
&\geq& 2|X|+1-2m
\\
&\geq& 2\times(2m+1)+1-2m
\\
&=& 2m+3
\\
&\geq& 5.
\end{eqnarray*}
Then we can choose three distinct vertices $u,v,w\in Sun(G-X)$ such that
$\{u,v,w\}$ is an independent set of $G$, and hence
$d_G(u)+d_G(v)+d_G(w)\geq\sigma_3(G)\geq n+2m$.
This together with $N_{G}(u)\cup N_{G}(v)\cup N_{G}(w)\subseteq X$ implies
$$
|X|\geq \max\{d_G(u),d_G(v),d_G(w)\}
\geq \frac{\sigma_3(G)}{3}
\geq \frac{n+2m}{3}.
$$
This together with the inequality $sun(G-X)\geq 2|X|+1-2m$ implies that
$n=|G|\geq|X|+sun(G-X)\geq3|X|+1-2m\geq n+1$, which is a contradiction.

\vspace{3mm}

{\bf Case~2}. $G-X$ is connected.

\vspace{2mm}

It is easily seen that $sun(G-X)\leq\omega(G-X)=1$ since $G-X$ is connected.
Recall that
\begin{equation}\label{eqn3-4}
sun(G'-X)=sun(G-E'-X)\leq sun(G-X)+2m.
\end{equation}
This together with (\ref{eqn3-2}) implies that $2|X|+1\leq sun(G'-X)\leq sun(G-X)+2m\leq 1+2m$, and thus
\begin{equation}\label{eqn3-5}
|X|\leq m.
\end{equation}
It follows from (\ref{eqn3-5}) and $\kappa(G)\geq2m+1$ that
\begin{equation}\label{eqn3-6}
\lambda(G-X)\geq\kappa(G-X)\geq\kappa(G)-|X|\geq2m+1-m=m+1.
\end{equation}
Using (\ref{eqn3-6}), we obtain that $\lambda(G'-X)\geq\lambda(G-X)-m\geq1$.
Hence, $sun(G'-X)\leq\omega(G'-X)=1<2|X|$, which contradicts (\ref{eqn3-2}).
\end{proof}

\section{Open problems}
A graph $G$ is called a $P_{\geq k}$-factor covered graph if it has a $P_{\geq k}$-factor covering $e$ for any $e\in E(G)$, where $k\geq2$ is an integer.
Zhang and Zhou \cite{Zhang2009} proposed the concept of path-factor covered graph, which is a generalization of matching covered graph.
They also obtained a characterization for $P_{\geq 2}$-factor and $P_{\geq 3}$-factor covered graphs, respectively.
Recently, Zhou and Sun \cite{Zhou2020} extended the concept of $P_{\geq k}$-factor covered graph to $(P_{\geq k},l)$-factor critical covered graph, namely, a graph $G$ is called $(P_{\geq k},l)$-factor critical covered if $G-D$ is $P_{\geq k}$-factor covered for any $D\subseteq V(G)$ with $|D|=l$.
Similar to $(P_{\geq k},l)$-factor critical covered graph, the concept of $(P_{\geq k},m)$-factor deleted graph can be further extended to $(P_{\geq k},m)$-factor deleted covered graph, that is, a graph $G$ is a $(P_{\geq k},m)$-factor deleted covered graph
if deleting any $m$ edges from $G$, the resulting graph is still a $P_{\geq k}$-factor covered graph.

We raise the following open problems as the end of our paper.

\begin{pro}
What is the tight $s(G)$ or $\sigma_3(G)$ bound of $(P_{\geq k},l)$-factor critical covered graphs?
\end{pro}
\begin{pro}
What is the tight $s(G)$ or $\sigma_3(G)$ bound of $(P_{\geq k},m)$-factor deleted covered graphs?
\end{pro}

\section*{Acknowledgements}
This work is supported by the National Natural Science Foundation of China (Grant Nos. 11971196,12201304,12201472) and Hainan Provincial Natural Science Foundation of China(No.120QN176).

\section*{Conflict of interests}
The authors declare that there are no conflict of interests.

\end{document}